\documentclass{amsart}
\usepackage[leqno]{amsmath}

\usepackage{amsmath}
\usepackage{amsthm}
\usepackage{amssymb}
\usepackage{amscd}
\usepackage{amsxtra}     
\usepackage{epsfig}
\usepackage{verbatim}
\usepackage[all]{xypic}

\theoremstyle{plain} 

\newtheorem{thm}{Theorem}[section]

\newtheorem{prop}[thm]{Proposition}

\newtheorem{cor}[thm]{Corollary}

\newtheorem{ques}[thm]{Question}

\theoremstyle{remark}

\numberwithin{equation}{section}

\newcommand{\ZZ}{\mathbb{Z}}
\newcommand{\QQ}{\mathbb{Q}}
\newcommand{\RR}{\mathbb{R}}

\newcommand{\tensor}{\otimes}

 \DeclareMathOperator{\Spec}{Spec}
 \DeclareMathOperator{\Sing}{Sing}
 
\DeclareMathOperator{\Reg}{R}

 \DeclareMathOperator{\im}{im}

\DeclareMathOperator{\divi}{div}

\DeclareMathOperator{\Cl}{Cl}
 \DeclareMathOperator{\G}{G}
 \DeclareMathOperator{\A}{A}

 \newcommand{\Min}{\textup{Min}}

\begin{document}

\bibliographystyle{plain}

\title{On injectivity of maps between Grothendieck groups induced by completion}
\author{Hailong Dao}
\address{Department of Mathematics, University of Utah,  155 South 1400 East, Salt Lake City,
UT 84112-0090, USA} \email{hdao@math.utah.edu} \maketitle
\begin{abstract} We give an example of a local normal domain $R$
such that the map of Grothendieck groups $\G(R) \to \G(\hat R)$ is
not injective. We also raise some questions about the kernel of that
map.

\end{abstract}

\section{Introduction}

Let $(R,m,k)$ be a local ring and $\hat R$ the $m$-adic completion
of $R$. Let $\mathcal{M}(R)$ be the category of finitely generated
$R$-modules. The Grothendieck group of finitely generated modules
over $R$ is defined as:
\begin{displaymath}
 \G(R)= \frac{\displaystyle \bigoplus_{M \in
\mathcal{M}(R)}{\ZZ [M]}}{\langle[M_2]-[M_1]-[M_3] \ | \ 0\to M_1
\to M_2 \to M_3 \to 0 \ \text{is exact}\rangle}
\end{displaymath}

In \cite{KK}, Kamoi and Kurano studied injectivity of the map $\G(R)
\to \G(\hat R)$ induced by flat base-change.  They showed that such
map  is the injective in the following cases  : 1)$R$ is Hensenlian,
2) $R$ is the localization at the irrelevant ideal of a positively
graded ring over a field, or, 3) $R$ has only isolated singularity.
Their results raise the question: Is the map between Grothendieck
group induced by completion  always injective?

In \cite{Ho1}, Hochster announced a counterexample to the above
question:

\begin{thm}\label{melex}
Let $k$ be a field. Let $R =
k[x_1,x_2,y_1,y_2]_{(x_1,x_2,y_1,y_2)}/(x_1x_2
-y_1x_1^2-y_2{x_2}^2)$. Let $P = (x_1,x_2)$ and $M=R/P$. Then $[M]$
is in the kernel of the map $\G(R) \to \G(\hat R)$. However $[M]
\neq 0$ in $\G(R)$.
\end{thm}

Hochster's example comes from the ``direct summand hypersurface" in
dimension $2$ and is not normal. He predicted that there is also an
example which is normal. The main purpose of this note is to provide
such an example. We have:

\begin{prop}\label{main}
Let $R = \RR[x,y,z,w]_{(x,y,z,w)}/(x^2+y^2-(w+1)z^2)$. $R$ is a
normal domain. Let $P = (x,y,z)$ and $M=R/P$. Then $[M]$ is in the
kernel of the map $\G(R) \to \G(\hat R)$. However $[M] \neq 0$ in
$\G(R)$.
\end{prop}
This will be proved in Section \ref{example}. We note that Kurano
and Srinivas has recently constructed an example of a local ring $R$
such that the map $\G(R)_{\QQ} \to \G(\hat R)_{\QQ}$ is not
injective (see \cite{KS}). The ring in their example is not normal,
and we do not know if a normal example exists in that context (i.e.
with rational coefficients).

In section \ref{kernel} we will discuss some questions on the kernel
of the map $\G(R) \to \G(\hat R)$.

We would like to thank Anurag Singh for telling us about this
question and for some inspiring conversations. We thank Melvin
Hochster for generously sharing his unpublished note \cite{Ho2},
which provided the key ideas for our example. We also thank the
referee for many helpful comments.

\section{Our example}\label{example}
We shall prove Proposition \ref{main}. First we need to recall some
classical results:
\begin{cor} (Swan, \cite{Sw}, Corollary 11.8)\label{2.1}
Let $k$ be a field of characteristic not $2$, $n>1$ an integer and
$R = k[x_1,...,x_n]/(f)$ where $f$ is a non-degenerate quadratic
form in  $k[x_1,...,x_n]$. Then $\G(R) = \ZZ\oplus \ZZ/2\ZZ$ if
$C_0(f)$, the even part of the Clifford algebra of $f$, is simple.
\end{cor}

\begin{prop}(Samuel, see \cite{Fo}, Proposition 11.5)\label{2.2}
Let $k$ be a field of characteristic not $2$ and $f$ be a
non-degenerate quadratic form in  $k[x_1,x_2,x_3]$. Let $R
=k[x_1,x_2,x_3]/(f)$. If $f=0$ has no non-trivial solution in $k$
then $\Cl(R) = 0$.
\end{prop}

\begin{prop}(Kamoi-Kurano)\label{2.3}
Let $S = \oplus_{n\geq 0} S_n$ be a graded ring over a field $S_0$
and $S_{+}= \oplus_{n > 0} S_n$. Let $A=S_{S_{+}}$. Then the map
$\G(S) \to \G(A)$ induced by localization is an isomorphism.
\end{prop}

\begin{proof}
See the proof of Theorem 1.5 (ii) in \cite{KK}.
\end{proof}

Proposition \ref{main} now follows from the following Propositions
(clearly, $R$ is normal, since the singular locus $V(x,y,z)$ has
codimension $2$):

\begin{prop}
$[\hat M]=0$ in $\G(\hat R)$.
\end{prop}

\begin{proof}
$\hat R = \RR[[x,y,z,w]]/(x^2+y^2-(w+1)z^2)$. We want to show that
$[\hat R/P\hat R] =0$ in $\G(\hat R)$. Let $\alpha = \sqrt{w+1}$
which is a unit in $\hat R$. Let $Q=(x,y-\alpha z)\hat R$. Then
clearly $Q$ is a height $1$ prime in $ \hat R$ and $P\hat R = Q +
(y+\alpha z)\hat R$. The short exact sequence:

$$ 0 \to \hat R/Q \to \hat R/Q \to \hat R/P\hat R \to 0$$
where the second map is the multiplication by $y+\alpha z$ shows
that $[\hat R/P\hat R]=0$ in $\G(\hat R)$.
\end{proof}

\begin{prop}
 $[M] \neq 0$ in $\G(R)$.
\end{prop}

\begin{proof}
It is enough to show that $[M_P] \neq 0$ in $\G(R_P)$. Let $K =
\RR(w)$ then $R_P \cong K[x,y,z]_{(x,y,z)}/(f)$ where $f =
x^2+y^2-(w+1)z^2$. Let $S = K[x,y,z]/(f)$. Clearly $f$ is a
non-degenerate quadratic form. Since the rank of $f$ is $3$, an odd
number, $C_0(f)$ is a simple algebra over $K$ (see, for example,
\cite{La}, Chapter 5, Theorem 2.4). By \ref{2.1} and \ref{2.3},
$\G(R_P) = \G(S) = \ZZ\oplus\ZZ/(2)$. We claim that $f$ has no
non-trivial solution in $K$. Suppose it has. Then by clearing
denominators, there are polynomials $a(w), b(w), c(w) \in \RR[w]$
such that
$$ a(w)^2 + b(w)^2 = (w+1)c(w)^2.$$
The degree of $ a(w)^2 + b(w)^2$ is always even. The degree of
$(w+1)c(w)^2$ is odd unless $c(w)=0$. But then $ a(w)^2 + b(w)^2 =0$
which forces $a(w)=b(w)=0$, a contradiction. By the claim and
\ref{2.2}, $\Cl(R_P) = \Cl(S)=0$. Thus $[R_P]$ and $[R_P/PR_P]$
generate $\G(R_P)=\ZZ\oplus\ZZ/(2)$ (since the Grothendieck group is
generated by $\{[R_P/Q], Q\in \Spec R_P\} $ and $\dim R_P =2$).
Since $\ZZ\oplus\ZZ/(2)$ can not be generated by one element,
$[R_P/PR_P]$ must be nonzero (it is easy to see that $[R_P/PR_P]$ is
$2$-torsion).
\end{proof}

\section{On the kernel of the map $\G(R) \to \G(\hat R)$}\label{kernel}

In this section we raise some questions about the kernel of the map
$\G(R) \to \G(\hat R)$. First we fix some notations. Throughout this
section we will assume, for simplicity, that $R$ is excellent, and
is  a homomorphic image of a regular local ring $T$. Let $d= \dim
R$. Let $\A_i(R)$ be the i-th Chow group of $R$, i.e.,

$$
\A_i(R) = \frac{\displaystyle \bigoplus_{P\in \Spec R, \dim R/P = i
} \ZZ \cdot [\Spec R/P] }{\langle \divi(Q,x) \ | \ Q\in \Spec R,
\dim R/Q = i+1 , \ x\in R\backslash Q \rangle}
$$
where
$$
\divi(Q,x) = \sum_{P \in \Min_R R/(Q,x)} l_{R_P}(R_P/(Q,x)R_P)[\Spec
R/P].
$$
For an abelian group $A$, we let $A_{\mathbb{Q}} =
A\tensor_{\mathbb{Z}}\mathbb{Q}$. The Chow group of $R$ is defined
to be $\A_*(R)=\displaystyle \oplus_{i=0}^{d} \A_i(R)$. It is well
known that there is a $\QQ$-vector space isomorphism:
$$ \tau_{R/T} : \G(R)_{\mathbb{Q}}\to \A_*(R)_{\mathbb{Q}}$$
(It is unknown whether this is independent of $T$). We also remark
that the Grothendieck group $\G(R)$ admits a filtration by $F_i\G(R)
= \langle [M] \in \G(R) \ | \ \dim M \leq i \rangle $.

The existing examples on the failure of injectivity for the map
$\G(R)\to \G(\hat R)$ and the affirmative results in \cite{KK}
motivate the following question:

\begin{ques}\label{kerGro}
Suppose that $R$ satisfies $(\Reg_n)$ (i.e, regular in codimension
$n$). Then is $\ker (\G(R) \to \G(\hat R))$  contained in
$F_{d-n-1}\G(R)$?
\end{ques}

Question \ref{kerGro} is closely related to the following:

\begin{ques}\label{kerChow}
Suppose that $R$ satisfies $(\Reg_n)$. Then is the map $\A_i(R) \to
\A_i(\hat R)$ injective for $i\geq d-n$?
\end{ques}

In fact, if we allow rational coefficients, then the previous
questions are equivalent. Let $\G^{i}(R) = F_{i}\G(R)/F_{i-1}\G(R)$.
Then clearly we have a decomposition:
 $$ \G(R)_{\QQ} =
\bigoplus_{i=0}^{d} \G^i(R)_{\QQ}$$ Also, the Riemann-Roch map
decomposes into isomorphisms $\tau^i: \G^i(R)_{\QQ} \to
A_i(R)_{\QQ}$, which make the following diagram:
\[
\xymatrix{ \G^i(R) \ar[d]^{g_i} \ar[r]^{\tau_{R/T}^i} & \A_i(R) \ar[d]^{f_i} \\
\G^i(\hat R) \ar[r]^{\tau_{\hat R/\hat T}^i} & \A_i(\hat R) }
\]
commutative. It follows that $$ \ker (\G(R)_{\QQ} \to \G(\hat
R)_{\QQ}) \cong \bigoplus_i^d \ker (f_i) \cong \bigoplus_i^d \ker
(g_i).$$ Thus we have:

\begin{prop}
Let $R$ be an excellent local ring which is a homomorphic image of a
regular local ring. Let $\dim R=d$ and let $0 < l\leq d$ be an
integer. Then the maps $\A_i(R)_{\QQ} \to \A_i(\hat R)_{\QQ}$ are
injective for $i\geq l$ if and only if $\ker (\G(R)_{\QQ} \to
\G(\hat R)_{\QQ}) \subseteq F_{l-1}\G(R)_{\QQ}$.
\end{prop}

We do not know if \ref{kerChow} is true even if $l=1$. Note that if
$R$ is normal, then both \ref{kerGro} and \ref{kerChow} are true for
$l=1$. In that situation $\A_1(R) \cong \Cl(R)$, and the map between
class groups of $R$ and $\hat R$ is injective. Furthermore, it is
well known that $\G(R)/F_{d-2}\G(R) \cong \A_d(R)\oplus\A_{d-1}(R)$
(see, for example \cite{Ch}, Corollary 1), so \ref{kerGro} is also
true for $l=1$ .

Finally, one could formulate a stronger version of \ref{kerGro} as
follows. Note that in both Hochster's example and the example
presented here, the support of the modules given actually equal to
the singular locus of $R$. So one could ask:

\begin{ques}
Let $R$ be an excellent local ring. Let $X=\Spec R$,  $Y =
\Sing(X)$, $\hat X = \Spec \hat R$ and $\hat Y = \Sing (\hat X)$.
One then has a commutative diagram:

\[
\xymatrix{ \G(Y) \ar[d] \ar[r]^{f} & \G(X) \ar[d]^{g} \\
\G(\hat Y) \ar[r] & \G(\hat X) }
\]
(Here $\G(X)$ denotes the Grothendieck group of coherent $\mathcal
O_X$-modules and the maps are naturally induced by closed immersions
or flat morphisms). Is $\ker(g)$ contained in $\im(f)$?
\end{ques}

\end{document}